\documentclass[12pt,a4paper]{article}

\usepackage{latexsym}
\usepackage{amsmath}
\usepackage{amssymb}
\usepackage{amsthm}
\usepackage{amscd}
\usepackage{mathrsfs}
\usepackage[all]{xy}
\usepackage{graphicx}
\usepackage{slashbox}
\usepackage{comment}

\newtheorem{thm}{Theorem}[section]
\newtheorem{prop}[thm]{Proposition}

\newtheorem{lem}[thm]{Lemma}

\newtheorem{rem}[thm]{Remark}

\theoremstyle{remark}

\makeatletter

\@addtoreset{equation}{section}
\makeatother

\makeatletter
\newcommand{\subsubsubsection}{\@startsection{paragraph}{4}{\z@}%
 {1.0\Cvs \@plus.5\Cdp \@minus.2\Cdp}%
 {.1\Cvs \@plus.3\Cdp}%
 {\reset@font\sffamily\normalsize}
 }
\makeatother
\DeclareMathOperator{\Gal}{Gal}

\DeclareMathOperator{\Aut}{Aut}

\DeclareMathOperator{\Tr}{Tr}

\newcommand{\bZ}{\mathbb{Z}}
\newcommand{\bQ}{\mathbb{Q}}

\newcommand{\bC}{\mathbb{C}}

\newcommand{\bF}{\mathbb{F}}

\newcommand{\bP}{\mathbb{P}}

\newcommand{\cE}{\mathcal{E}}

\newcommand{\cO}{\mathcal{O}}

\newcommand{\iGL}{\mathit{GL}}
\newcommand{\iSL}{\mathit{SL}}

\begin{document}

\title
{Local root numbers of elliptic curves\\ over dyadic fields}
\author{Naoki Imai} 

\date{}

\maketitle

\begin{abstract}
We consider an elliptic curve over a dyadic field 
with additive, potentially good reduction. 
We study the finite Galois extension of the dyadic field 
generated by the three-torsion points of the elliptic curve. 
As an application, 
we give a formula to calculate the local root number 
of the elliptic curve over the dyadic field. 
\end{abstract}

\section*{Introduction}
Let $K$ be a non-archimedean local field 
with residue field $k$. 
Let $E$ be an elliptic curve over $K$. 
If $E$ has potentially multiplicative reduction, 
then $E$ has split multiplicative reduction 
over a quadratic extension of $K$ 
(cf.~Proposition \ref{multquad}). 
On the other hand, 
if $E$ has potentially good reduction, 
then we need a bigger extension 
to get good reduction in general. 

We assume that $E$ has potentially good reduction. 
Let $p$ be the characteristic of $k$. 
We consider a 
finite Galois extension $L$ of $K$, 
which is obtained by adding the 
coordinates of the $(p+1)$-torsion points of $E$. 
Then $E$ has good reduction over $L$ (cf.~Proposition \ref{pgext}). 
The inertia subgroup of $\Gal (L/K)$ 
is studied by Kraus in \cite{Kraelladd} 
if the characteristic of $K$ is zero. 
We will extend the results to 
positive characteristic cases. 
Actually, if $p \geq 3$, the proof in \cite{Kraelladd} 
works without change. 
Hence, we focus on the 
case where $p=2$. 
If $p=2$, then $K$ is called a dyadic field. 
Further, we study the Galois group itself 
not only the inertia subgroup. 

The local root numbers of elliptic curves 
are studied by many people. 
If the characteristic of $K$ is zero, 
they are calculated by Rohrlich in 
\cite{RohGalellroot} 
except the case where 
$p=2, 3$ and the elliptic curves have 
additive potentially good reduction. 
Halberstadt gives a table of local root numbers 
of elliptic curves 
in \cite{HalSign23} 
if $K$ is $\bQ_2$ or $\bQ_3$. 
Kobayashi calculates them in \cite{Kobrootell} 
in the case where $p \geq 3$ and 
$E$ has potentially good reduction. 
Using a result of \cite{Kraelladd}, 
Dokchitser-Dokchitser shows a formula 
calculating the local root numbers of elliptic curves 
over $2$-adic fields in \cite{DoDoroot2}. 
In this paper, 
we extend the formula of 
Dokchitser-Dokchitser to the positive characteristic cases 
using the study of the Galois extension $L$ over $K$. 
See Theorem \ref{rootf} for details of the formula. 

In Section \ref{ellcur}, we recall basic facts on 
elliptic curves over non-archimedean local fields. 
In Section \ref{Galgrp}, we study the Galois group 
$\Gal (L/K)$ and give a classification in Theorem \ref{clGal}. 
In Section \ref{rootnum}, 
we show a formula calculating 
the root numbers of elliptic curves over dyadic fields. 
The proof of the formula is rather independent of 
the classification in Theorem \ref{clGal}. 

\subsection*{Acknowledgment}
The author is grateful to a referee for 
careful reading and 
a number of suggestions for improvements. 

\subsection*{Notation}
In this paper, we use the following notation. 
Let $K$ be a non-archimedean local field 
with a residue field $k$ of characteristic $p$. 
We write $\mathcal{O}_K$ for the 
ring of integers in $K$. 
Let $v$ be the normalized valuation of $K$. 
We take an algebraic closure $K^{\mathrm{ac}}$ 
of $K$. 
For any finite extension $F$ of $K$, let 
$W_F$ denote the Weil group of $F$. 

\section{Elliptic curve}\label{ellcur}
Let $E$ be an elliptic curve over $K$. 
Then $E$ has a minimal Weierstrass equation 
\[
 y^2 +a_1 xy +a_3 y=x^3 +a_2 x^2 +a_4 x +a_6 
\]
with $a_1, \ldots, a_4 ,a_6 \in \cO_K$. 
We put 
\begin{align*}
 b_2 &=a_1 ^2 +4 a_2 , \\ 
 b_4 &=a_1 a_3 +2 a_4 , \\ 
 b_6 &=a_3 ^2 +4a_6 , \\ 
 b_8 &=a_1 ^2 a_6 -a_1 a_3 a_4 +4a_2 a_6 +a_2 a_3^2 -a_4^2 , \\ 
 c_4 &=b_2 ^2 -24 b_4 , \\ 
 c_6 &=-b_2 ^3 +36 b_2 b_4 -216 b_6 , \\ 
 \Delta &=-b_2^2 b_8 -8b_4^3 -27b_6^2 +9b_2 b_4 b_6 , \\ 
 j &= c_4^3 /\Delta. 
\end{align*}
Then we have 
\begin{align}
 4b_8 &=b_2 b_6 -b_4^2 , \label{b8264} \\ 
 1728 \Delta &=c_4^3 -c_6^2. \label{Dc46}
\end{align}
as in \cite[(1.3)]{TatAlgsf}. 
Then $E$ has potentially good reduction 
if and only if $v(j) \geq 0$ 
(cf.~\cite[VII. Proposition 5.5]{SilAEC}). 

The following fact is due to Tate: 

\begin{prop}[{cf.~\cite[Appendix C. Theorem 14.1.(d)]{SilAEC}}]\label{multquad}
If $v(j) < 0$, 
then $E$ has split multiplicative reduction 
over a quadratic extension of $K$. 
\end{prop} 

For a finite extension $F$ over $K$ 
inside $K^{\mathrm{ac}}$ 
and 
a set $S$ of points of $E(K^{\mathrm{ac}})$, 
let $F(S)$ be the extension of $F$ inside $K^{\mathrm{ac}}$ 
obtained by 
adding the $x$ and $y$ coordinates of 
the points of $S$ to $F$. 
For such an $F$ and 
a point $P$ of $E$, 
we simply write $F(P)$ for $F(\{ P \})$. 
For a positive integer $m$, 
let $E[m]$ denote the kernel of 
the $m$-multiplication map on $E(K^{\mathrm{ac}})$. 

\begin{prop}\label{pgext}
Let $m \geq 3$ be an integer that is prime to $p$. 
If $v(j) \geq 0$, then 
$E$ has good reduction over 
$K(E[m])$. 
\end{prop}
\begin{proof}
This follows from 
\cite[Corollary 3 to Theorem 2]{SeTaGred} 
and 
\cite[VII. Proposition 5.4.(a)]{SilAEC}. 
\end{proof}

\section{Group theory}\label{Grpth}
For a natural number $n$, 
we write 
$\mathfrak{S}_n$ for the symmetric group of degree $n$, 
$C_n$ for the cyclic group of order $n$, 
$D_{2n}$ for the dihedral group of order $2n$, 
$\mathit{SD}_{2^n}$ for the semidihedral group of order $2^n$. 
We write $Q_8$ for the quaternion group. 
Here we recall an elementary fact on group theory. 

\begin{prop}\label{sgGL}
The natural action of $\iGL_2 (\bF_3)$ 
on $\bP^1 (\bF_3 )$ defines a 
surjection $\iGL_2 (\bF_3) \to \mathfrak{S}_4$. 
Furthermore, $2$-$2$ partitions of the $4$ point set 
$\{ 1, 2, 3, 4 \}$ defines a surjection 
$\mathfrak{S}_4 \to \mathfrak{S}_3 =D_6$. 
For a subgroup of $\iGL_2 (\bF_3)$, 
we consider the image and kernel of the restriction 
of $\iGL_2 (\bF_3) \to D_6$ to the subgroup. 
Then a list of 
the isomorphism classes of subgroups of $\iGL_2 (\bF_3)$ 
is given by the following table: 
\begin{center}
\begin{tabular}{|c|c|c|c|c|} \hline
\backslashbox{$\mathrm{Ker}$}{$\mathrm{Im}$} 
& $C_1$ & $C_2$ & $C_3$ & $D_6$ \\ \hline
$C_1$ & $C_1$ & $C_2$ & $C_3$ & $D_6$ \\ \hline
$C_2$ & $C_2$ & $C_2 \times C_2$ & $C_6$ & $D_{12}$ \\ \hline
$C_4$ & $C_4$ & $C_8$, $D_8$ & - & - \\ \hline
$Q_8$ & $Q_8$ & $\mathit{SD}_{16}$ & 
$\iSL_2 (\bF_3)$ & $\iGL_2 (\bF_3)$ \\ \hline
\end{tabular} @
\end{center}
Further, the images under the surjection 
$\iGL_2 (\bF_3) \to \mathfrak{S}_4$ 
of subgroups that are isomorphic to $C_8$ and $D_8$  
are isomorphic to $C_4$ and $C_2 \times C_2$ respectively.  
\end{prop}
\begin{proof}
We treat only distinction between 
$C_8$, $D_8$ and $Q_8$. 
It is well-known that 
the kernel of the surjection $\iGL_2 (\bF_3) \to D_6$ 
is the unique subgroup that is isomorphic to $Q_8$. 
The subgroups that are isomorphic to 
$C_8$ are conjugate to the subgroup generated by 
\[
\begin{pmatrix}
0 & 1\\ 
1 & 1
\end{pmatrix}.\] 
Hence, their images in $\mathfrak{S}_4$ 
are isomorphic to $C_4$. 
On the other hand, 
the subgroups that are isomorphic to 
$D_8$ are conjugate to the subgroup generated by 
\[
\begin{pmatrix}
1 & 0\\ 
0 & -1
\end{pmatrix} \quad \textrm{and} \quad 
\begin{pmatrix}
0 & 1\\ 
1 & 0
\end{pmatrix}. 
\]
Hence, their images in $\mathfrak{S}_4$ 
are isomorphic to $C_2 \times C_2$. 
\end{proof}

\section{Galois group}\label{Galgrp}
We assume that $p \neq 3$. 
Let $E$ be an elliptic curve over $K$ with 
additive, potentially good reduction. 
We put $L=K(E[3])$ and $G=\Gal (L/K)$. 
Then $E$ has good reduction over $L$ 
by Proposition \ref{pgext}. 

We put 
\[
 g(x) =3x^4 +b_2 x^3 + 3b_4 x^2 +3b_6 x+b_8 . 
\]
Let $\alpha_1, \alpha_2, \alpha_3 ,\alpha_4$ 
be the roots of $g(x)$ in $K^{\mathrm{ac}}$. 
\begin{prop}\label{f3D}
(1) The $x$ coordinates of the $8$ non-trivial points of 
$E[3]$ are the roots of $g(x)$. \\ 
(2) The set of the third roots of $\Delta$ is 
\[
 \bigl\{ b_4 -3(\alpha_1 \alpha_2 + \alpha_3 \alpha_4 ),\ 
 b_4 -3(\alpha_1 \alpha_3 + \alpha_2 \alpha_4 ),\ 
 b_4 -3(\alpha_1 \alpha_4 + \alpha_2 \alpha_3 ) \bigr\} .
\]
(3) The degree $[L:K]$ is not divided by $3$ 
if and only if $\Delta \in (K^{\times})^3$. 
\end{prop}
\begin{proof}
These are proved in \cite[5.3.b)]{Serpfell}. 
\end{proof}

We take $\Delta^{1/3} \in K^{\mathrm{ac}}$ 
so that 
$\Delta^{1/3} \in K$ if 
$\Delta \in (K^{\times})^3$. 
We assume that 
$\Delta^{1/3} =b_4 -3(\alpha_1 \alpha_2 + \alpha_3 \alpha_4 )$ 
by renumbering $\alpha_1, \alpha_2, \alpha_3 ,\alpha_4$. 
We put $K_1 =K(\Delta^{1/3})$, 
$s=\alpha_1 +\alpha_2$ and 
$t=\alpha_1 \alpha_2$.

\begin{lem}\label{fdec}
(1) We have 
\[
 g(x) = 
 (x^2 -s x +t ) 
 \bigl( 
 3x^2 +(3s +b_2 )x -(3t + \Delta^{1/3} -b_4 ) 
 \bigr). 
\]
(2) We have $[K_1 (s,t) : K_1] \leq 2$. 
\end{lem}
\begin{proof}
The claim (1) follows from 
$b_2 =-3(\alpha_1 +\alpha_2 +\alpha_3 +\alpha_4 )$ and 
$\Delta^{1/3} =b_4 -3(\alpha_1 \alpha_2 + \alpha_3 \alpha_4 )$. 

We take a basis of $E[3]$. 
Then we have an embedding $G \hookrightarrow \iGL_2 (\bF_3)$, 
and the roots of $g(x)$ correspond 
the elements of $\bP^1 (\bF_3 )$. 
By considering $\mathfrak{S}_4$ as the automorphism group of 
$\{ \alpha_1, \alpha_2, \alpha_3 ,\alpha_4 \}$, 
we have a surjection 
$\iGL_2 (\bF_3) \to \mathfrak{S}_4$ 
as in Proposition \ref{sgGL}. 

Then $K_1$ corresponds to the stabilizer in $G$ of 
the partition 
$\{ \{\alpha_1, \alpha_2 \}, \{\alpha_3, \alpha_4 \} \}$, 
which is the intersection of $G$ and a subgroup of $\iGL_2 (\bF_3)$ 
that is isomorphic to $\mathit{SD}_{16}$. 
On the other hand, 
$K_1 (s,t)$ corresponds to the stabilizer in $G$ of 
the subset 
$\{\alpha_1, \alpha_2 \}$, 
which is the intersection of $G$ and a subgroup of $\iGL_2 (\bF_3)$ 
that is isomorphic to $D_8$. Therefore, we have 
$[K_1 (s,t) : K_1] \leq \lvert \mathit{SD}_{16}/D_8 \rvert = 2$. 
\end{proof}

We take a non-trivial third root of unity $\zeta_3 \in K^{\mathrm{ac}}$. 

\begin{lem}
We have $\zeta_3, \Delta^{1/3} \in L$. 
\end{lem}
\begin{proof}
We have $\zeta_3 \in L$ by the existence of the Weil pairing, 
and $\Delta^{1/3} \in L$ by 
$\Delta^{1/3} =b_4 -3(\alpha_1 \alpha_2 + \alpha_3 \alpha_4 )$. 
\end{proof}

The following lemma and proposition are variants of 
results in \cite{Kraelladd}. 
The proofs in \cite{Kraelladd} work also in our situation. 
We recall the proofs for completeness. 

\begin{lem}[{cf.~\cite[Lemme 5]{Kraelladd}}]\label{div3}
Let $P$ be a non trivial element of $E[3]$. 
Then $[L:K(\zeta_3 ,P)]$ divides $3$. 
\end{lem}
\begin{proof}
We take a point $Q$ of $E[3]$ so that 
$(P, Q )$ is an ordered basis of $E[3]$. 
Then it gives an injective group homomorphism 
$\rho \colon G \to \iGL_2 (\bF_3 )$. 
Then the image of 
$\Gal \bigl( L/K(\zeta_3 ,P) \bigr)$ under $\rho$ 
is contained in 
\[
 \biggl\{ 
 \begin{pmatrix}
 1 & b \\ 
 0 & 1 
 \end{pmatrix}
 \biggm| 
 b \in \bF_3 \biggr\}, 
\]
because $\det \rho$ is the mod $3$ cyclotomic character. 
Hence, we have the claim. 
\end{proof}

\begin{prop}[{cf.~\cite[Proposition 3]{Kraelladd}}]\label{eqro} 
Let $K'$ be an extension of 
$K_1 ( \zeta_3 )$ contained in $L$. 
Then the followings are equivalent: 
\begin{enumerate}
\renewcommand{\labelenumi}{(\arabic{enumi})}
\item 
$[L:K'] \leq 2$. 
\item 
$g(x)$ has a root in $K'$. 
\item 
$g(x)$ has the all roots in $K'$. 
\end{enumerate}
\end{prop}
\begin{proof}
It is trivial that (3) implies (2). 
We assume (2). 
Let $\alpha$ be a root of $g(x)$ in $K$. 
Let $P$ be a point of $E[3]$ whose 
$x$ coordinate is $\alpha$. 
Then we have $[L:K'(P)]=1$ by 
Proposition \ref{f3D}.(3) and Lemma \ref{div3}. 
Hence we have (1). 

We assume (1). 
Taking a basis of $E[3]$, we have 
an injective group homomorphism 
$\rho \colon G \to \iGL_2 (\bF_3 )$. 
Then the image of 
$\Gal (L/K')$ under $\rho$ 
is contained in 
\[
 \biggl\{ 
 \begin{pmatrix}
 1 & 0 \\ 
 0 & 1 
 \end{pmatrix}
 , 
 \begin{pmatrix}
 -1 & 0 \\ 
 0 & -1 
 \end{pmatrix}
 \biggr\}. 
\]
Then, the all roots of $g(x)$ is fixed by 
the action of $\Gal (L/K')$. 
Hence, we have (3). 
\end{proof}

The following lemma is also a variant of a lemma in \cite{Kraelladd}. 
Our proof is different from that in \cite{Kraelladd}. 

\begin{lem}[{cf.~\cite[Lemme 6]{Kraelladd}}]
Let $\alpha_0 \in K^{\mathrm{ac}}$ be a root of $g(x)$. 
Then $K_1 (\zeta_3, \alpha_0 )$ contains 
$s$ and $t$. 
\end{lem}
\begin{proof}
By Lemma \ref{eqro}, 
$K_1 (\zeta_3, \alpha_0 )$ contains the 
all roots of $g(x)$. 
Hence the claim follows. 
\end{proof}

We simply write $\alpha$ for $\alpha_1$, 
and take $\beta \in K^{\mathrm{ac}}$ such that 
\begin{equation*}
 \beta^2 +a_1 \alpha \beta +a_3 \beta 
 =\alpha^3 +a_2 \alpha^2 +a_4 \alpha +a_6. 
\end{equation*}

\begin{lem}\label{Lgen}
We have $L=K_1 (\zeta_3, \alpha, \beta)$. 
\end{lem}
\begin{proof}
By Lemma \ref{div3}, $[L:K_1 (\zeta_3, \alpha, \beta)]$ 
divides $3$. 
On the other hand, $3$ does not divide $[L:K_1]$ 
by Proposition \ref{f3D}.(2). 
Hence the claim follows. 
\end{proof}

\begin{thm}\label{clGal}
(1) Suppose $\zeta_3 \in K$ and $\Delta^{1/3} \in K$. 
\begin{enumerate}
\renewcommand{\labelenumi}{(\alph{enumi})}
\item 
If $\alpha \in K$, 
then $G \simeq C_2$. 
\item 
If $\alpha \notin K$ and $s ,t \in K$, 
then $G \simeq C_4$. 
\item 
If $K(s ,t) \neq K$ and $\alpha \in K(s ,t)$, 
then $G \simeq C_4$. 
\item 
If $K(s ,t) \neq K$  and $\alpha \notin K(s ,t)$, 
then $G \simeq Q_8$. 
\end{enumerate} 
(2) Suppose $\zeta_3 \in K$ and $\Delta^{1/3} \notin K$. 
\begin{enumerate}
\renewcommand{\labelenumi}{(\alph{enumi})}
\item 
If $\alpha, \beta \in K_1$, 
then $G \simeq C_3$. 
\item 
If $\alpha \in K_1$ and $\beta \notin K_1$, 
then $G \simeq C_6$. 
\item 
If $K_1 (s ,t) \neq K_1$, 
then $G \simeq \iSL_2 (\bF_3 )$. 
\end{enumerate}
(3) Suppose $\zeta_3 \notin K$ and $\Delta^{1/3} \in K$. 
\begin{enumerate}
\renewcommand{\labelenumi}{(\alph{enumi})}
\item 
If $\alpha \in K(\zeta_3 )$, 
then $G \simeq C_2 \times C_2$. 
\item 
If $\alpha \notin K(\zeta_3 )$, 
$s ,t \in K(\zeta_3 )$ and $K(s ,t) \neq K$, 
then $G \simeq C_8$. 
\item 
If $\alpha \notin K(\zeta_3 )$ and $s ,t \in K$, 
then $G \simeq D_8$. 
\item 
If $K(\zeta_3, s ,t) \neq K(\zeta_3)$ and 
$\alpha \in K(\zeta_3, s ,t)$, 
then $G \simeq D_8$. 
\item 
If $K(\zeta_3 ,s ,t) \neq K(\zeta_3 )$, 
then $\alpha \notin K(\zeta_3 )$ and $G \simeq \mathit{SD}_{16}$. 
\end{enumerate} 
(4) Suppose $\zeta_3 \notin K$ and $\Delta^{1/3} \notin K$. 
\begin{enumerate}
\renewcommand{\labelenumi}{(\alph{enumi})}
\item 
If $\alpha, \beta \in K_1 (\zeta_3 )$, 
then $G \simeq D_6$. 
\item 
If $\alpha \in K_1 (\zeta_3 )$ and 
$\beta \notin K_1 (\zeta_3 )$, 
then $G \simeq D_{12}$. 
\item 
If $K_1 (\zeta_3 , s ,t) \neq K_1 (\zeta_3 )$, 
then $\alpha \notin K_1 (\zeta_3 )$ 
and $G \simeq GL_2 (\bF_3 )$. 
\end{enumerate}
\end{thm}
\begin{proof}
By taking a basis of $E[3]$, 
we consider $G$ as a subgroup of $\iGL_2 (\bF_3)$. 
We note that (1), (2), (3) and (4) in this theorem 
correspond the 1st, 3rd, 2nd and 4th column 
in Proposition \ref{sgGL} respectively. 
We use Proposition \ref{sgGL} without mention.

We prove (1). 
Since $E$ has bad reduction, 
the claim (1a) follows from Lemma \ref{eqro}. 
If $\alpha \notin K$, then $[L:K] \geq 4$ by 
Proposition \ref{f3D}.(3) and Lemma \ref{eqro}. 
On the other hand, if $s ,t \in K$, then 
$[L:K] \leq 4$ by Lemma \ref{Lgen}. 
Hence the claim (1b) follows. 
By applying Lemma \ref{eqro} to $K(s ,t)$, 
we have the claims (1c) and (1d). 

We prove (2). The claims (2a) and (2b) follows from 
Lemma \ref{Lgen}. 
If $K_1 (s ,t) \neq K_1$, 
$[L:K_1] \geq 4$ by (1c) and (1d). 
Hence we have (2c), because 
there is no subgroup of $\iSL_2 (\bF_3)$ 
with order $12$. 

We prove (3). By replacing 
$K$ by $K(\zeta_3 )$ in (1), 
we have (3a), (3e), and $[L:K]=8$ in the case 
(3b), (3c) and (3d). 
In the case (3b), 
$g(x)$ is irreducible by Lemma \ref{fdec}.(1). 
Then $G$ act transitively on the roots of $g(x)$. 
Hence we have (3b). 
In the case (3c), 
we have two distinct quadratic extensions 
$K(\zeta_3)$ and $K(\alpha )$ of $K$. 
In the case (3d), 
we have two distinct quadratic extensions 
$K(\zeta_3)$ and $K(s ,t )$ of $K$. 
Hence we have (3c) and (3d). 

By replacing 
$K$ by $K(\zeta_3 )$ in (2), 
we obtain (4). 
\end{proof}

\begin{rem}
In some case, 
the inertia subgroup $I$ of $G$ is determined by 
the Kodaira-N\'{e}ron type of $E$. 
In fact, 
$I \simeq \bZ/3\bZ$ if and only if 
the Kodaira-N\'{e}ron type of $E$ is $\mathit{IV}$ or $\mathit{IV}^*$, 
where we use the Kodaira symbol after \cite{KodStrsI}. 
This fact can be proved similarly as 
\cite[Th\'{e}or\`{e}me 2]{Kraelladd} 
also in the positive characteristic case. 
\end{rem}

\section{Root number}\label{rootnum}
We assume that $K$ is a dyadic field. 
Let $\phi \colon \bF_2 \to \bC^{\times}$ 
be the non-trivial character. 
We take an additive character 
$\psi \colon K \to \bC^{\times}$ 
such that 
$\psi (a) =\phi (\Tr_{k/\bF_2} (\bar{a} ))$ 
for $a \in \cO_K$, 
where $\bar{a}$ denotes the image of $a$ 
in $k$. 
Let $d\mu$ be a Haar measure on $K$. 
Let $\sigma$ be 
a finite dimensional smooth representation 
of $W_K$ over $\bC$. 
Then we can consider a 
local $\epsilon$-factor 
$\epsilon (\sigma ,\psi ,d\mu ) \in \bC^{\times}$ 
as in \cite[\S 4]{Delcef}. 
We put 
\[
 w(\sigma, \psi) = 
 \frac{\epsilon (\sigma, \psi, d\mu )}{ \lvert \epsilon 
 (\sigma ,\psi ,d\mu ) \rvert }, 
\]
which is independent of the choice of $d\mu$. 
We call $w(\sigma, \psi)$ the local root number 
of $\sigma$ with respect to $\psi$. 
For a finite extension $F$ over $K$ and 
a finite dimensional smooth representation 
of $W_F$ over $\bC$, 
we always consider the root number with respect to 
$\psi \circ \Tr_{F/K}$. 
We simply write 
$w(\sigma)$ for $w(\sigma, \psi)$ in the sequel. 

Let $T_3 (E)$ be the $3$-adic Tate module of $E$. 
We put $V_3 (E) =T_3 (E) \otimes_{\bZ_3} \bQ_3$. 
We take an embedding $\bQ_3 \to \bC$. 
Then the natural action of $W_K$ on $V_3 (E)$ induces 
a smooth representations 
\[
 \sigma_E \colon W_K \to 
 \Aut (V_3 (E) \otimes_{\bQ_3} \bC ), 
\]
because $E$ has potentially good reduction. 
We put 
$w(E/K)=w(\sigma_E )$, 
which is called the local root number of $E$. 

Using results in Section \ref{Galgrp}, 
we can extend results in \cite{DoDoroot2} 
to positive characteristic cases. 
Here, we treat only the most non-trivial case, 
where $G \simeq \iGL_2 (\bF_3)$. 

\begin{rem}\label{cpsi}
Our choice of $\psi$ is different from that in 
\cite{DoDoroot2}. 
It is the reason why the formulas in 
Lemma \ref{twist} and Theorem \ref{rootf} 
look different from those in \cite{DoDoroot2}. 
\end{rem}

We assume that $G \simeq \iGL_2 (\bF_3)$. 
We put $f=[k : \bF_2]$ and 
$n(E)=v(\Delta)$. 
We note that 
$f$ is odd and 
$K(\zeta_3 )$ is the unramified quadratic extension 
by the assumption. 
Let $\eta \colon W_K \to \bC^{\times}$ 
be the unramified character that 
sends the arithmetic Frobenius to 
$(\sqrt{2}i)^f$. 
We put 
$\sigma_{E,\eta} =\sigma \otimes \eta^{-1}$. 

\begin{lem}[{cf.~\cite[Lemma 1]{DoDoroot2}}]\label{twist}
The $W_K$-representation $\sigma_{E,\eta}$ 
factors through $G$. 
Moreover, we have 
\[
 w (E/K) =(-i)^{f(n(E) -2)} w(\sigma_{E,\eta}). 
\]
\end{lem}
\begin{proof}
The proof in \cite[Lemma 1]{DoDoroot2} 
works also in our situation. 
\end{proof}

\begin{lem}\label{discr}
The discriminant of $g(x)$ is 
equal to $-27 \Delta^2$. 
\end{lem}
\begin{proof}
Using Proposition \ref{f3D}.(2), 
we see that 
the discriminant of $g(x)$ is equal to 
\begin{align*}
 3^6 \prod_{1 \leq i < j \leq 4} & (\alpha_i -\alpha_j )^2 \\ 
 &= (\Delta^{1/3} - \zeta_3 \Delta^{1/3} )^2 
 (\Delta^{1/3} - \zeta_3^2 \Delta^{1/3} )^2 
 (\zeta_3 \Delta^{1/3} - \zeta_3^2 \Delta^{1/3} )^2 
 =-27 \Delta^2. 
\end{align*}
\end{proof}

We put 
\begin{align*}
 s' &= 
 \zeta_3 (\alpha_1 +\alpha_2 ) +\zeta_3^2 (\alpha_3 +\alpha_4 ), \\ 
 t' &= 
 \zeta_3 \alpha_1 \alpha_2 +\zeta_3^2 \alpha_3 \alpha_4 \\ 
\end{align*}
and $M=K_1 (s',t')$.

\begin{lem}\label{GalC8}
We have an isomorphism 
$\Gal (L/M) \simeq C_8$. 
\end{lem}
\begin{proof}
Let $h_0$ be the element of 
$\mathfrak{S}_4 =\Aut (\{ \alpha_1, \alpha_2, \alpha_3 ,\alpha_4 \} )$ 
defined by 
\[
 \alpha_1 \mapsto \alpha_3  \mapsto \alpha_2 
 \mapsto \alpha_4 \mapsto \alpha_1. 
\]
Let $H$ be the preimage of the subgroup generated by $h_0$ 
under $G \to \mathfrak{S}_4$. 
Then $H$ is isomorphic to $C_8$ by Proposition \ref{sgGL}. 
Any lift of $h_0$ in $G$ send 
$\zeta_3$ to $\zeta_3^2$, 
because it fixes $\Delta^{1/3}$ and permutes 
$\zeta_3 \Delta^{1/3}$ and $\zeta_3^2 \Delta^{1/3}$. 
Hence, $H$ fixes $s'$ and $t'$ by the definition. 
This implies that 
$H \subset \Gal (L/M)$ and $[M:K_1] \leq 2$. 
Therefore, it suffices to show $M \neq K_1$. 

Since $K_1(s,t) \neq K_1$ by the assumption 
$G \cong \iGL_2 (\bF_3 )$, 
either 
\[
 \alpha_1 +\alpha_2 \neq \alpha_3 +\alpha_4 \quad 
 \textrm{or} \quad 
  \alpha_1 \alpha_2 \neq \alpha_3 \alpha_4 
\]
holds. 
Hence, $M$ is not fixed by $\Gal (L/K_1)$. 
This show $M \neq K_1$. 
\end{proof}

Let $\cE$ be the elliptic curve over $\bF_2$ 
defined by $x^3 =y^2 +y$. 
The following fact is well-known: 
\begin{lem}[{cf.~\cite{ITreal3}}]\label{ellrat}
Let $m$ be a positive integer. 
Then we have 
$\sharp \cE (\bF_{2^m}) 
 =2^m +1 -(\sqrt{2}i)^m -(-\sqrt{2}i)^m$. 
\end{lem}

By Lemma \ref{twist}, 
we consider 
$\Sigma_{E,\eta}$ as a representation of $G$. 
We take a character 
$\chi \colon \Gal (L/M) \to \bC^{\times}$ 
such that $\Sigma_{E,\eta}|_{\Gal (L/M)}$ 
is the direct sum of $\chi$ and its conjugate. 
For a character $\phi$ of a subgroup of $G$, 
let $w(\phi)$ denote 
the root number of the character of the Weil group 
corresponding to $\phi$. 
For a finite extension $F$ of $K$ and 
its quadratic extension $F'$, 
let $w_{F'/F}$ 
be the root number of 
the non-trivial character of $W_F$ 
that factors through $\Gal (F'/F)$. 
The following is a main theorem, which is proved at 
\cite[Theorem 7]{DoDoroot2} in the case where 
$K$ is of characteristic $0$. 

\begin{thm}\label{rootf}
We have 
\[
 w (E/K) =(-i)^{f(n(E) -2)}
 \frac{w (\chi)}{w_{K(\alpha, \beta)/K(\alpha)}}, 
\]
where $f$ and $n (E)$ are defined after Remark \ref{cpsi}. 
\end{thm}
\begin{proof}
We take an ordered basis $(P,Q)$ of $E[3]$ such that 
the coordinate of $P$ is $(\alpha,\beta)$. 
We identify $G$ with $\iGL_2 (\bF_3)$ by this ordered basis. 
By Lemma \ref{twist}, 
it suffices to show 
\begin{equation}\label{wchi2}
 w(\sigma_{E,\eta}) = 
 \frac{w (\chi)}{w_{K(\alpha, \beta)/K(\alpha)}}. 
\end{equation}
Let $B_{12} \subset \iGL_2 (\bF_3)$ 
be the $\bF_3$-rational points 
of the upper triangle Borel subgroup of $\iGL_2$. 
Let $\det_{B_{12}}$ be the determinant character, and 
$\sigma_{B_{12}}$ be the non-trivial character that 
factors through 
\[
 B_{12} \to \bF_3^{\times} ; \ 
 \begin{pmatrix}
 a & b \\ 0 & d 
 \end{pmatrix}
 \mapsto a. 
\]
We take a non-trivial character $\tau$ of 
$\iSL (\bF_3)$. 
(We note that the abelianization of $\iSL (\bF_3)$ 
is isomorphic to $\bZ/3\bZ$.) 
Let $k_2$ be the residue field of $K(\zeta_3)$. 
By the local class field theory, 
$\tau$ corresponds to a non trivial character 
$\chi_{\tau}$ of $k_2^{\times}/(k_2^{\times})^3$. 
Then the formula 
\begin{equation}\label{wchiBtau}
 w(\sigma_{E,\eta}) = 
 \frac{w(\chi) w(\det_{B_{12}}) w(\tau)}{w(\sigma_{B_{12}})} 
\end{equation}
is proved in the proof of \cite[Theorem 7]{DoDoroot2} 
without using the assumption that 
the characteristic of $K$ is $0$. 
Since $B_{12}$ is the subgroup of $G$ 
preserving the subgroup of $E[3]$ generated by $P$, 
we have 
$w(\sigma_{B_{12}}) =w_{K(\alpha, \beta)/K(\alpha)}$. 
We have 
\[
 w({\det}_{B_{12}})=w_{K(\alpha,\zeta)/K(\alpha)} = 
 (-1)^{[K(\alpha):K]-v(\mathfrak{D}(K(\alpha)/K))} =1, 
\] 
where we have the second equality by 
\cite[23.5. Lemma 1 and Proposition]{BHLLCGL2}, and 
the third equality holds since 
$v(\mathfrak{D}(K(\alpha)/K))$ is even by Lemma \ref{discr}. 
We put $q=p^f$. 
We have 
\begin{align}
 w(\tau) &= q^{-1} \sum_{x \in k_2^{\times}} 
 \chi_{\tau} (x)^{-1} \phi (\Tr_{k_2/\bF_2} (x)) \notag \\ 
 &= q^{-1} \Biggl\{ 
 \sum_{x \in (k_2^{\times})^3} 
 \phi (\Tr_{k_2/\bF_2} (x)) 
 - \sum_{x \in \zeta_3 (k_2^{\times})^3} 
 \phi (\Tr_{k_2/\bF_2} (x)) 
 \Biggr\} \label{S-S}
\end{align}
using \cite[23.5. Theorem]{BHLLCGL2} and that 
\[
 \bigl\{ x \in \zeta_3 (k_2^{\times})^3 \bigm| 
 \Tr_{k_2/\bF_2} (x)=0 \bigr\} 
 \to 
 \bigl\{ x \in \zeta_3^2 (k_2^{\times})^3 \bigm| 
 \Tr_{k_2/\bF_2} (x)=0 \bigr\}; \ 
 x \mapsto x^2 
\]
is a bijection. 
We put 
\begin{align*}
 N_1 &=\sharp \bigl\{ y \in k_2 \bigm| 
 y+y^2 \in (k_2^{\times})^3 \bigr\}, \\ 
 N_2 &=\sharp \bigl\{ y \in k_2 \bigm| 
 y+y^2 \in \zeta_3 (k_2^{\times})^3 \bigr\}. 
\end{align*}
We note that $N_2 =(q^2 -2-N_1)/2$. 
Then \eqref{S-S} is equal to 
\[
 q^{-1} \Biggl\{ 
 \biggl(
 \frac{N_1}{2} - \Bigl( 
 \frac{q^2 -1}{3} -\frac{N_1}{2} \Bigr) 
 \biggr) 
 - 
 \biggl( 
 \frac{N_2}{2} -\Bigl( 
 \frac{q^2 -1}{3} -\frac{N_2}{2} \Bigr) 
 \biggr) \Biggr\} = 
 \frac{3N_1 -q^2 +2}{2q} 
 = 1, 
\]
because $N_1 =(q^2 +2q -2)/3$ by Lemma \ref{ellrat}. 
Therefore \eqref{wchi2} follows from \eqref{wchiBtau}. 
\end{proof}

\begin{rem}
The elliptic curve $\cE$ appears in a semi-stable reduction of 
a Lubin-Tate curve over a dyadic field (cf.~\cite{ITstab3}). 
Hence, it is studied in \cite{ITreal3}. 
Actually, 
a similar calculation as 
the calculation of $w(\tau)$ in the proof of 
Theorem \ref{rootf} 
appears in \cite{ITreal3}. 
\end{rem}


\noindent
Naoki Imai\\ 
Graduate School of Mathematical Sciences, 
The University of Tokyo, 3-8-1 Komaba, Meguro-ku, 
Tokyo, 153-8914, Japan\\ 
naoki@ms.u-tokyo.ac.jp\\


\end{document}